\documentclass[12pt,reqno]{amsart}
  \usepackage{latexsym} 
  \usepackage[all]{xy}
  \usepackage{amsfonts} 
  \usepackage{amsthm} 
  \usepackage{amsmath} 
  \usepackage{amssymb}
  \usepackage{pifont}  
  \usepackage{enumerate}
  \usepackage{dcolumn}
  \usepackage{comment}
\usepackage{hyperref}  
\newcolumntype{2}{D{.}{}{2.0}}
  \xyoption{2cell}

 \usepackage{tikz}
\usetikzlibrary{arrows}

   \def\<{{\langle}} 
  \def\>{{\rangle}} 
   
  \def\la{{\triangleright}}

  \def\note#1{{}}

  \def\note#1{}

  \def\beq{\begin{equation}} 
  \def\eeq{\end{equation}}

  \def\id{\mathrm{id}}

  \newcounter{zlist} 
  \newenvironment{zlist}{\begin{list}{(\arabic{zlist})}{ 
  \usecounter{zlist}\leftmargin2.5em\labelwidth2em\labelsep0.5em 
  \topsep0.6ex 
  \parsep0.3ex plus0.2ex minus0.1ex}}{\end{list}}

  \newcounter{blist} 
  \newenvironment{blist}{\begin{list}{(\alph{blist})}{ 
  \usecounter{blist}\leftmargin2.5em\labelwidth2em\labelsep0.5em 
  \topsep0.6ex 
  \parsep0.3ex plus0.2ex minus0.1ex}}{\end{list}} 

  \newcounter{rlist} 
  \newenvironment{rlist}{\begin{list}{(\roman{rlist})}{ 
  \usecounter{rlist}\leftmargin2.5em\labelwidth2em\labelsep0.5em 
  \topsep0.6ex
  \parsep0.3ex plus0.2ex minus0.1ex}}{\end{list}}

\def\stac#1{\raise-.2cm\hbox{$\stackrel{\displaystyle\otimes}{\scriptscriptstyle{#1}}$}}

\def\cten#1{\raise-.2cm\hbox{$\stackrel{\displaystyle\reallywidehat{\otimes}}
{\scriptscriptstyle{#1}}$}}

  \def\Label#1{\label{#1}\ifmmode\llap{[#1] }\else 
  \marginpar{\smash{\hbox{\tiny [#1]}}}\fi} 
  \def\Label{\label}

  \newtheorem{proposition}{Proposition}[section]
  \newtheorem{lemma}[proposition]{Lemma} 
  \newtheorem{corollary}[proposition]{Corollary} 
  \newtheorem{theorem}[proposition]{Theorem} 
  
\theoremstyle{definition} 
  \newtheorem{definition}[proposition]{Definition}
    
  \newtheorem{example}[proposition]{Example}

  \theoremstyle{remark} 
  \newtheorem{remark}[proposition]{Remark}

  \newcounter{c} 
   
  \newcommand{\etyk}[1]{\vspace{-7.4mm}$$\begin{equation}\Label{#1} 
  \addtocounter{c}{1}} 
  \renewcommand{\]}{\ifnum \value{c}=1 $$\else \end{equation}\fi} 
   \numberwithin{equation}{section}

\def\CC{{\mathbb C}}

\def\FF{{\mathbb F}}

\def\NN{{\mathbb N}}

\def\RR{{\mathbb R}}

\def\aaa{{\mathfrak a}}
\def\bbb{{\mathfrak b}}

\def\ggg{{\mathfrak g}}
\def\hhh{{\mathfrak h}}

\def\lll{{\mathfrak l}}

\def\nnn{{\mathfrak n}}
\def\ooo{{\mathfrak o}}

\def\sss{{\mathfrak s}}

\newcommand{\Cc}{\mathcal{C}}

\def\*C{{}^*\hspace*{-1pt}{\Cc}}

\def\text#1{{\rm {\rm #1}}}

 \def\1{\mathbf{1}}

      \def\tr{\mathrm{Tr}\ }

\def\id{\mathrm{id}}

\def\1\mathbf{1}

\def\|#1{\overline{#1}}

\def\ad{\mathrm{ad}}
\def\Tan#1#2{T_{#1}{#2}}
\usepackage{scalerel,stackengine}
\stackMath
\newcommand\reallywidehat[1]{%
\savestack{\tmpbox}{\stretchto{%
  \scaleto{%
    \scalerel*[\widthof{\ensuremath{#1}}]{\kern.1pt\mathchar"0362\kern.1pt}%
    {\rule{0ex}{\textheight}}
  }{\textheight}%
}{2.4ex}}%
\stackon[-6.9pt]{#1}{\tmpbox}%
}
\def\Lin{\mathrm{Lin}}

\begin{document}

\title{Lie affgebras vis-\`a-vis Lie algebras}

\author[Andruszkiewicz]{Ryszard R.\ Andruszkiewicz}

\address{
Faculty of Mathematics, University of Bia{\l}ystok, K.\ Cio{\l}kowskiego  1M,
15-245 Bia\-{\l}ys\-tok, Poland}

\email{randrusz@math.uwb.edu.pl}

\author[Brzeziński]{Tomasz Brzezi\'nski}

\address{
Department of Mathematics, Swansea University,
Fabian Way,
  Swansea SA1 8EN, U.K.\ \newline \indent
Faculty of Mathematics, University of Bia{\l}ystok, K.\ Cio{\l}kowskiego  1M,
15-245 Bia\-{\l}ys\-tok, Poland}

\email{T.Brzezinski@swansea.ac.uk}

\author[Radziszewski]{Krzysztof Radziszewski}

\address{Doctoral School of Exact and Natural Sciences \&
Faculty of Mathematics, University of Bia{\l}ystok, K.\ Cio{\l}kowskiego,
15-245 Bia\-{\l}ys\-tok, Poland}

\email{K.Radziszewski@uwb.edu.pl}

\subjclass[2010]{}

\keywords{Lie algebra; Lie affgebra; generalised derivation; quasicentroid}

\begin{abstract}
It is shown that any Lie affgebra, that is an algebraic system consisting of an affine space together with a bi-affine bracket satisfying affine versions of the antisymmetry and Jacobi identity, is isomorphic to a Lie algebra together with an element and a specific generalised derivation (in the sense of Leger and Luks, [G.F.\ Leger \& E.M.\ Luks, Generalized derivations of Lie algebras, {\em J.\ Algebra} {\bf 228} (2000), 165--203]). These Lie algebraic data can be taken for the construction of a Lie affgebra or, conversely, they can be uniquely derived for any Lie algebra fibre of the Lie affgebra. The close relationship between Lie affgebras and (enriched by the additional data) Lie algebras can be employed to attempt a classification of the former by the latter. In particular, up to isomorphism, a complex Lie affgebra with a simple Lie algebra fibre $\mathfrak{g}$ is fully determined by a scalar and an element of $\mathfrak{g}$ fixed up to an automorphism of $\mathfrak{g}$, and it can be universally embedded in a trivial extension of $\mathfrak{g}$ by a derivation. The study is illustrated by a number of examples that include all Lie affgebras with one-dimensional, nonabelian two-dimensional, $\mathfrak{s}\mathfrak{l}(2,\mathbb{C})$ and  $\mathfrak{s}\mathfrak{o}(3)$ fibres. Extensions of Lie affgebras by cocycles and their relation to cocycle extensions of tangent Lie algebras is  briefly discussed too.
\end{abstract}    
\date\today
\maketitle

\section{Introduction}
The study of vector space valued Lie brackets on affine spaces or {\em Lie affgebras} was initiated in \cite{GraGra:Lie} and developed and applied to the investigation of  differential geometry of bundles with affine fibres or {\em AV-geometry} \cite{GraGra:av1}, \cite{GraGra:av2} with an eye to the frame independent formulation of Lagrangian mechanics \cite{Tul:fra}, \cite{GraGra:fra}. A new point of view and extension of Lie affgebras in which no reference to a vector space is made or, indeed, no existence of the underlying vector space is presumed, was proposed in \cite{BrzPap:Lie}. In this approach a vector space is an artefact rather than a fundamental ingredient of an affine space in the sense that any point of an affine space determines a vector space (with the chosen point playing the role of the zero vector), the {\em tangent space} or the {\em vector fibre} at this point. The  linearisation of the bi-affine Lie bracket on an affine space on any vector fibre determines the Lie algebra structure on this fibre. Thus a Lie affgebra can be seen as a {\em Lie algebra fibred affine space}.  The aim of the present text is  to reveal and explore the relationship between Lie affgebras and the corresponding Lie algebra structures on the fibres.

The bi-affine map property of the Lie bracket on an affine space yields the distributivity of the bracket over the ternary heap operation arising from the translation of a point by a vector between two other points. The results of \cite{AndBrzRyb:ext} show that in the associative case there is a very clear connection between abelian heaps with bi-heap associative operation or {\em trusses} and extensions of rings by integers: trusses are in one-to-one correspondence with a specific class of extensions of rings by {\em double homotheties} of Redei \cite{Red:ver} or {\em self-permutable bimultiplications} of MacLane \cite{Mac:ext}.  Due to different natures of extensions in the associative and Lie cases, one cannot expect such a straightforward correspondence between Lie algebra extensions and Lie affgebras except for exeptional cases (such as derivation Lie affgebras discussed in Section~\ref{sec.idem}). On the other hand one can ask the following three questions. First, does the fact that a Lie algebra is a Lie fibre of a Lie affgebra yield additional data on it? Second, are there specific data which make a Lie algebra a fibre of a Lie affgebra or, in other words, can one construct a unique Lie affgebra from a Lie algebra with additional structure? Third, do all Lie affgebras arise in that way from Lie algebras (with additional data)? We give affirmative and constructive answers to all these questions. 

The paper is organised as follows. In Section~2 we recall and slightly modify the definition of a Lie affgebra from \cite{BrzPap:Lie}. The main difference from \cite{BrzPap:Lie} is that the Jacobi identity is now homogenised. Section~\ref{sec.LieLie} contains  main results. First,  Theorem~\ref{thm.alg.aff}  states that any Lie affgebra $\ggg$ determines and is determined by a Lie algebra (which is isomorphic to the Lie fibre of the Lie affgebra) together with two linear maps $\lambda$ and $\kappa$, and a constant. There is no restriction on the constant, while $\lambda$ is a {\em generalised derivation} in the sense of \cite{LegLuk:gen} of the type $(\delta, \lambda,\lambda)$, with $\delta=\lambda-\kappa$. This necessarily means that $\kappa$ is an element of the {\em quasicentroid} of $\ggg$, $\mathrm{QC}(\ggg)$, the Lie sub-algebra of linear endomorphisms of  $\ggg$ extensively studied in \cite{LegLuk:gen}.  Second, Theorem~\ref{thm.hom} asserts that a homomorphism between Lie affgebras is given by a homomorphism between Lie fibres and a constant. The Lie algebra homomorphism intertwines the respective elements of the quasicentroids, while intertwining general derivations up to the adjoint action by the constant. This allows for the formulation of clear criteria for isomorphisms between Lie affgebras. 

In Section~\ref{sec.idem} we analyse Lie affgebras in which defining linear transformations $\lambda$ and $\kappa$ of $\ggg$ combine into a (no longer generalised) derivation $\delta =\lambda-\kappa$. This happens if and only if $\kappa$ is an element of the {\em centroid} of $\ggg$, $\mathrm{C}(\ggg)$, a Lie sub-algebra of $\mathrm{QC}(\ggg)$ studied in depth in \cite[Chapter~X]{Jac:Lie}.  Consequently, the combination of results of \cite[Chapter~X]{Jac:Lie}, \cite{LegLuk:gen} and \cite{BurDek:pos} yields that over an algebraically closed field and up to isomorphism a  Lie affgebra with a simple Lie algebra fibre $\mathfrak{g}$ is fully determined by a scalar and an element of $\mathfrak{g}$, the latter given up to an automorphism of $\ggg$.  In bigger generality, when $\kappa$ is a scalar multiple of the identity,  we extend the existence of {\em Lie hulls} of \cite{GraGra:Lie}, i.e.\ Lie algebras in which Lie affgebras are universally embedded as cosets of ideals, beyond the Lie affgebras with idempotent brackets (which bijectively correspond to Lie affgebras of Grabowska, Grabowski and Urbański \cite{GraGra:Lie}). The construction of Lie affgebras from Lie algebras described in Section~\ref{sec.LieLie} is illustrated by a number of examples in Section~\ref{sec.ex}. We classify all Lie affgebras with one-dimensional fibres, non-abelian two-dimensional fibres and $\mathfrak{s}\mathfrak{l}(2,\mathbb{C})$ and  $\mathfrak{s}\mathfrak{o}(3)$ fibres. In Section~\ref{sec.ex.gna} we demonstrate how the correspondence between Lie affgebras and Lie algebras can be used to conclude non-isomorphism of a class of matrix Lie affgebras over finite fields, thus justifying the necessity of the restriction of the characteristic of the field or equivalently the size of the matrices demanded in \cite[Theorem~3.2]{BrzRad:mat}. The paper is concluded with  Section~\ref{sec.appl} in which we discuss {\em cocycle extensions} of Lie affgebras and relate them to central extensions of their Lie algebra fibres. 

\subsection*{Notation and conventions.} We work over a field $\FF$ of characteristic different from 2. The space of linear  endomorphisms of a vector space $V$ is denoted by $\mathrm{Lin}(V)$.

 The adjoint action of $a\in \ggg$ on a Lie algebra $\ggg$ is denoted by $\ad_a:\ggg\to\ggg$, $b\mapsto  [a,b]$, where  $[-,-]$ is the Lie bracket.

\section{Lie brackets on affine spaces}\label{sec.prem}
By an {\em abelian heap} \cite{Bae:ein}, \cite{Pru:the} we mean an algebraic system consisting of a set $X$ together with a ternary operation $\langle - ,-,-\rangle: X^3\to X$ such that, for all $x_i\in X$, $i=1,\ldots, 5$,
$$
\langle x_1 ,x_2,x_3\rangle = \langle x_3 ,x_2,x_1\rangle, \; \langle x_1 ,x_1,x_2\rangle = x_2, \; \langle \langle x_1 ,x_2,x_3\rangle, x_4,x_5\rangle = \langle  x_1 ,x_2,\langle  x_3 ,x_4,x_5\rangle\rangle.
$$
As a consequence of these conditions, the distribution of brackets does not matter, and hence we write simply
$
\langle x_1 ,x_2,x_3, x_4,x_5\rangle$ for $\langle  x_1 ,x_2,\langle  x_1 ,x_2,x_3\rangle\rangle.$
A homomorphism of heaps is a function $f: X\to Y$ preserving the operations in the sense that, for all $x_i\in X$, $f(\langle x_1 ,x_2,x_3\rangle) = \langle f(x_1) ,f(x_2),f(x_3)\rangle$. By fixing an element $o\in X$ and reducing ternary operation of a non-empty heap to a binary operation $x+y := \langle x ,o,y\rangle$, one obtains an abelian group (the retract of $X$ at $o$). Conversely, any abelian group determines a unique abelian heap structure with the operation $\langle a,b,c\rangle = a-b+c$. A heap homomorphism $f: X\to Y$ uniquely defines the homomorphism of abelian groups between the respective retracts  by the assignment $x\mapsto f(x) - f(o)$.

\begin{definition}\label{def.aff}
Following \cite{OstSch:bar} (see also \cite[Definition~4.8 \& Proposition~4.9]{BreBrz:hea}) by an {\em affine space} over a field $\FF$ or an {\em $\FF$-affine space} we mean an algebraic system consisting of a non-empty set $\aaa$ and two ternary operations

$$
\begin{aligned}
\langle -, - ,- \rangle : \aaa^3 \to \aaa, \qquad & (a,b,c)\mapsto \langle a, b ,c \rangle,\\
-\la_- -: \FF\times \aaa^2\to \aaa, \qquad & (\alpha, a,b) \mapsto \alpha \la_a b,
\end{aligned}
$$
such that,
\begin{blist}
   \item  $(\aaa, \langle -, - ,- \rangle)$ is an abelian heap;
    \item for all $a,b\in \aaa$ and $\alpha\in \FF$, 
    $
    \alpha \la_a-: \aaa\to \aaa $ 
    and  $-\la_a b: \FF\to \aaa$
    are 
    homomorphisms of heaps, where $\FF$ is the heap with the operation $\alpha-\beta+\gamma$;
    \item for all $a\in \aaa$, the map $-\la_a - : \FF\times \aaa \to \aaa$ is an action of the multiplicative monoid of $\FF$, that is, for all $\alpha,\beta \in \FF$ and $b\in \aaa$, $(\alpha\beta)\la_ab = \alpha\la_a(\beta\la_a b)$ and $1\la_a b=b$;
    \item for all $a,b\in \aaa$, $0\la_ab = a$;
    \item the {\em base change property} is satisfied, that is, for all $\alpha \in \FF$ and $a,b,c\in \aaa$,
    $$
    \alpha \la_a b = \langle \alpha \la_c b , \alpha \la_c a, a\rangle.
    $$
\end{blist} 
    For a given $a\in \aaa$, the map $-\la_a -: \FF\times \aaa\to\aaa$ is called an {\em affine action} (of  $\FF$ on $\aaa$) {\em with the  base} $a$.

    An {\em affine map} $f:\aaa\to\bbb$ is a heap homomorphism preserving the actions in the sense that, for all $a,b\in \aaa$ and $\alpha \in \FF$,
    $$
    f\left(\alpha\la_a b\right) = \alpha \la_{f(a)} f(b).
    $$
    The set of affine maps from $\aaa$ to $\bbb$ is denoted by $\mathrm{Aff}(\aaa,\bbb)$.
\end{definition}
\begin{remark}\label{rem.aff}
   Given an element $o$ of an $\FF$-affine space $\aaa$, the retract of the heap $\aaa$ at $o$, i.e., the abelian group structure on $\aaa$ with addition $a+b = \langle a, o, b\rangle$ and the neutral element $o$, together with the operation
   \begin{equation}\label{scalar.m}
       \FF\times \aaa\to \aaa, \qquad (\alpha, a)\mapsto \alpha a := \alpha\,\la_oa,
   \end{equation}
   is a vector space. We call it the {\em tangent space to $\aaa$ at $o$} or the {\em vector space fibre of $\aaa$ at $o$} and denote it by $\Tan o\aaa$. Different choices of $o$ lead to isomorphic tangent spaces. The vector space $\Tan o\aaa$ acts freely and transitively on $\aaa$ by 
   $$
   \aaa \times \Tan o\aaa \to \aaa, \qquad (a,b)\mapsto \langle a,o,b\rangle.
   $$
Thanks to the base change property (e) in Definition~\ref{def.aff}, the affine action of $\aaa$ in terms of the linear structure of its fibre at $o$ \eqref{scalar.m} can be written as
\begin{equation}\label{scalar.a}
    \alpha\, \la_ab = (1-\alpha)a +\alpha\, b,
\end{equation}
for all $a,b\in \aaa$ and $\alpha \in \FF$. (This formula is independent of the choice of $o$.) On the other hand, any vector space can be understood as an affine space with the affine action \eqref{scalar.a}.

Every $f\in \mathrm{Aff}(\aaa,\bbb)$ induces the unique linear transformation $\hat{f}: \Tan o\aaa \to \Tan{\tilde{o}}\aaa$, by $\hat{f}: a\mapsto f(a) - f(o)$. Conversely, given any linear transformation $\hat{f}: \Tan o\aaa \to \Tan{\tilde{o}}\aaa$ and an element $b\in \bbb$, the map $\aaa\to\bbb$, $a\mapsto \hat{f}(a) + b$ is an affine transformation.  All elements of $\mathrm{Aff}(\aaa,\bbb)$ arise in that way.

   The preceding discussion clarifies the connection between the traditional definition of an affine space as a set accompanied with a vector space (of free vectors) acting on it and Definition~\ref{def.aff}. 
\end{remark}

\begin{remark}\label{rem.aff.extra}
    The conditions of Definition~\ref{def.aff} imply that, for all $\alpha \in \FF$ and $b\in \aaa$, the map $\alpha \la_- b : \aaa \to \aaa$, $a \mapsto \alpha \la_a b$, is a homomorphism of heaps (so there is no need to state it explicitly as a part of the definition). Furthermore,
    \begin{equation}\label{idem}
      \alpha \la _a a = a,  
    \end{equation}
    for all $\alpha\in \FF$ and $a\in \aaa$. For the proofs see e.g.\ \cite[Lemma~3.5]{BreBrz:hea}.
\end{remark}

The main object of studies of this paper is given in the following
\begin{definition}\label{def.Lie}
    Let $\aaa$ be an affine space over $\FF$. A {\em Lie bracket} on $\aaa$ is a binary operation $\{-,-\}:\aaa\times \aaa \to \aaa$ that satisfies the following conditions: 
    \begin{blist}
    \item for all $a\in \aaa$, both $\{a,-\}$ and $\{-,a\}$ are affine transformations $\aaa\to \aaa$;
    \item affine antisymmetry, that is, for all $a,b\in \aaa$,
    \begin{equation}\label{antisym}
        \langle \{a,b\},\{a,a\}, \{b,a\}\rangle = \{b,b\};
    \end{equation}
    \item the affine Jacobi identity, that is, for all $a,b,c\in \aaa$,
    \begin{equation}\label{Jacobi}
        \langle \{a,\{b,c\}\},\{a,\{a,a\}\}, \{b,\{c,a\}\}, \{b,\{b,b\}\}, \{c,\{a,b\}\} \rangle = \{c,\{c,c\}\}.
    \end{equation}
    \end{blist}
    Following the terminology introduced in \cite{GraGra:Lie}, an affine space together with a Lie bracket is called a {\em Lie affgebra}.

    A homomorphism of Lie affgebras is an affine map that preserves Lie brackets.
\end{definition}

\begin{remark}\label{rem.new.old}
     A Lie affgebra of Definition~\ref{def.Lie} is a slight modification of the notion introduced in \cite[Definition~3.1]{BrzPap:Lie} in so far as the inhomogeneous Jacobi identity
   \begin{equation}\label{Jacobi.old}
        \langle \{a,\{b,c\}\},\{a,a\}, \{b,\{c,a\}\}, \{b,b\}, \{c,\{a,b\}\} \rangle = \{c,c\}.
    \end{equation}
    there is now replaced by its homogeneous version \eqref{Jacobi}. Note that \eqref{Jacobi.old} implies that, for all $a\in \aaa$, necessarily $3\{a,a\} = 3\{a,\{a,a\}\}$ in any abelian group retract of the heap $(\aaa, \langle -,-,-\rangle)$. Thus if the characteristic of the field is different from three, \eqref{Jacobi.old} implies \eqref{Jacobi}.

    Needless to say, as in \cite[Definition~3.1]{BrzPap:Lie}, there is a version of \eqref{Jacobi} in which the nested brackets $\{-,\{-,-\}\}$ are replaced by $\{\{-,-\},-\}$, but clearly results obtained in one convention are easily translated to the other one.
\end{remark}

The key fact motivating this paper that can also be considered as a strong support for Definition~\ref{def.Lie} is an  observation made in \cite[Theorem~3.15]{BrzPap:Lie} that any tangent space to a Lie affgebra inherits a natural Lie algebra structure. This observation is not affected by the homogenisation of the Jacobi identity. 
\begin{theorem}\label{thm.tangentLie}
    Let $\aaa$ be a Lie affgebra with a bracket $\{-,-\}$. Then, for all $o\ in \ \aaa$, $\Tan o\aaa$ is a Lie algebra with the bracket
    \begin{equation}\label{tangentLie}
        [a,b] = \{a,b\} - \{a,o\} +\{o,o\} - \{o,b\},
    \end{equation}
    for all $a,b\in \Tan o\aaa$. We call $\Tan o\aaa$ with this bracket the {\em Lie algebra tangent to $\aaa$ at $o$} or simply a {\em tangent Lie algebra} or a {\em Lie algebra fibre} of $\aaa$. 
\end{theorem}
\begin{proof}
    Note that the bracket \eqref{tangentLie} is the unique bilinear map associated to the bi-affine map $\{-,-\}$ by the step by step linearisation that connects affine maps with the linear ones on the retracts as described in Remark~\ref{rem.aff}. In view of \eqref{antisym}, for all $a\in \aaa$, 
    $$
    [a,a] = \{a,a\} - \{a,o\} +\{o,o\} - \{o,a\} = o,
    $$
    and hence the bracket \eqref{tangentLie} is antisymmetric by bilinearity. 
    The Jacobi identity is proven by a straightforward direct calculation that uses the Jacobi identity \eqref{Jacobi} written in $\Tan o \aaa $,
    \begin{equation}\label{Jacobi.tan}
    \{a,\{b,c\}\}- \{a,\{a,a\}\} + \{b,\{c,a\}\} - \{b,\{b,b\}\} + \{c,\{a,b\}\} - \{c,\{c,c\}\} = o,
    \end{equation}
    for all $a,b,c\in \aaa$.
\end{proof}

All tangent Lie algebras to a given Lie affgebra are mutually isomorphic (as Lie algebras).

\begin{example}\label{ex.Lie.Lie}
    Similarly to vector spaces that can also be seen as affine spaces, any Lie algebra $\ggg$ is a Lie affgebra with the brackets, for all $a,b,s\in \ggg$,
    \begin{equation}\label{Lie.Lie}
        \{a,b\} = [a,b]+b +s.
    \end{equation}
    Any tangent Lie algebra to this Lie affgebra is isomorphic to $\ggg$. These are not all Lie affgebra structures with tangent $\ggg$. For example, given any $\zeta \in \FF$, an affine space $\aaa$ and $\chi\in \mathrm{Aff}(\aaa,\aaa)$, one can define the Lie affgebra $\aaa_{\zeta,\chi}$ with the affine action bracket 
    \begin{equation}\label{action.bracket.kappa}
        \{a,b\} = \langle\zeta\la _a b,a,\chi(a)\rangle.
    \end{equation}
If the characteristic of a field is different from 3 and $\zeta\neq 0$, then one easily checks that the bracket \eqref{action.bracket.kappa} satisfies the non-homogeneous affine Leibniz rule  if and only if $\chi = \id$ in which case the bracket takes the form
\begin{equation}\label{action.bracket}
        \{a,b\} = \zeta\la _a b.
    \end{equation}

As explained in \cite{BrzPap:Lie}, if $\aaa$ has at least two elements, $\aaa_{\zeta,\id} \cong \aaa_{\xi,\id}$ if and only if $\zeta =\xi$. On the other hand, for all $\zeta\in \FF$ and $o\in \aaa$, $\Tan o \aaa_{\zeta,\chi}$ is an abelian Lie algebra, i.e.\ $[a,b] = o$, for all $a,b\in \Tan o \aaa_{\zeta,\chi}$. 

It is one of the aims of this paper to study in detail the relationship between Lie affgebras and their tangent Lie algebras.
\end{example}

\section{Lie affgebras with a prescribed tangent Lie algebra}\label{sec.LieLie}

The first main result of this paper reveals close connection between Lie affgebras and generalised derivations of Lie algebras in the sense of Leger and Luks \cite{LegLuk:gen}.

\begin{theorem}\label{thm.alg.aff}
  Let $\ggg$  be a Lie algebra and $\kappa,\lambda\in \mathrm{Lin}(\ggg)$ be such that, for all $a,b\in \ggg$,
  \begin{equation}\label{kl}
      \lambda\left([a,b]\right) = [\lambda(a),b]-[a, \kappa(b)] + [a,\lambda(b)].
  \end{equation}
  Then, for all $s\in \ggg$, $\ggg$ is a Lie affgebra with the affine space structure
  $$
  \langle a,b,c \rangle = a-b+c, \qquad \alpha\la_{a} b = \alpha b+(1-\alpha)a, \quad \mbox{for all $a,b,c\in \ggg$, $\alpha\in \FF$,}
  $$
  and the affine Lie bracket, for all $a,b\in \ggg$,
  \begin{equation}\label{aff.bra}
      \{a,b\} = [a,b] +\kappa(a) + \lambda(b-a) + s.
  \end{equation}
We denote this Lie affgebra by $\aaa(\ggg;\kappa,\lambda,s)$. Furthermore, for all $o\in \ggg$
$$
\Tan o \aaa(\ggg;\kappa,\lambda,s) \cong \ggg.
$$

Conversely, for any Lie affgebra $\aaa$ and any $o\in \aaa$, there exist $\kappa,\lambda,s$ necessarily satisfying \eqref{kl} and such that 
$\aaa = \aaa(\Tan o\aaa;\kappa,\lambda,s)$. 
\end{theorem}

\begin{proof} Let $\ggg$ be a Lie algebra and $\kappa,\lambda$ and $s$ be as in the assumptions of theorem.  Due to the fact that the Lie bracket $[-,-]$ is a bilinear operation and $\kappa,\lambda$ are linear functions, the bracket \eqref{aff.bra} is a bi-affine transformation as the sum of a linear and constant parts.
To prove affine antisymmetry let us take any $a,b \in \ggg$, then from the antisymmetry of the Lie bracket and from the linearity of $\lambda$ we obtain
$$\begin{aligned}
    \langle \{a,b\},\{a,a\},\{b,a\} \rangle & = \{a,b\}-\{a,a\}+\{b,a\} \\
     & =
    [a,b]+\kappa(a)+\lambda(b-a)+s\\
     &-\kappa(a)-s+[b,a]+\kappa(b)+\lambda(a-b)+s\\
    & = 
    \kappa(b)+s=\{b,b\}
    \end{aligned}
    $$
    Thus the affine antisymmetry holds with no restrictions on the data $\lambda,\kappa$ and $s$.
    
    To show the affine Jacobi identity \eqref{Jacobi} of the bracket \eqref{aff.bra}, let us first note that
$$
\begin{aligned}
    \{a,\{b,c\}\}- \{a,\{a,a\}\} =& [a,[b,c]] + [a,\kappa(b)] + [a,\lambda(c)]
    - [a,\lambda(b)]\\
    &+\lambda\left([b,c]\right) + \lambda\kappa(b-a) +\lambda^2(c-b) - [a,\kappa(a)]. 
\end{aligned}
$$
Hence the bracket defined by the formula \eqref{aff.bra} satisfies the affine Jacobi identity \eqref{Jacobi} if and only if
$$\begin{aligned}
    0 = &\langle \{a,\{b,c\}\}, \{a,\{a,a\}\},\{b,\{c,a\}\},\{b,\{b,b\}\},\{c,\{a,b\}\} \rangle - \{c,\{c,c\}\} \\
    =& \; \{a,\{b,c\}\}- \{a,\{a,a\}\}+\{b,\{c,a\}\}-\{b,\{b,b\}\}+\{c,\{a,b\}\} \rangle - \{c,\{c,c\}\}\\
    =&\; [a,[b,c]] + [a,\kappa(b)]+[a,\lambda(c)]-[a,\lambda(b)]+\lambda([b,c]) - [a,\kappa(a)]\\
    &+[b,[c,a]] +[b,\kappa(c)]+[b,\lambda(a)]-[b,\lambda(c)]+\lambda([c,a])- [b,\kappa(b)]\\
    &+[c,[a,b]] +[c,\kappa(a)]+[c,\lambda(b)]-[c,\lambda(a)]+\lambda([a,b])- [c,\kappa(c)]\\
    =&\;  [a,\kappa(b)]+[a,\lambda(c)]-[a,\lambda(b)]+\lambda([b,c]) - [a,\kappa(a)] +[b,\kappa(c)]+[b,\lambda(a)]-[b,\lambda(c)]\\
    &+\lambda([c,a])- [b,\kappa(b)]
    +[c,\kappa(a)]+[c,\lambda(b)]-[c,\lambda(a)]+\lambda([a,b])- [c,\kappa(c)].
\end{aligned}$$
For $b=c$  the above equality holds if and only if 
$[b-a,\kappa(b-a)] =0$ or, equivalently, 
\begin{equation}\label{k}
   [a,\kappa(a)] =0, 
\end{equation}
for all $a\in \ggg$. Next, setting $c=0$  and in view of \eqref{k}, we deduce the necessity  of  \eqref{kl}. Since \eqref{k} follows from \eqref{kl}, by setting $b=a$ in \eqref{kl}, the bracket \eqref{aff.bra} satisfies the affine Jacobi identity.

Writing $[-,-]'$ for the bracket in $\Tan o\aaa(\ggg;\kappa,\lambda,s)$, in terms of the vector space structure of $\ggg$, we have
$$
[a,b]' = [a,b] - [a,o] - [o,b] +o,
$$
and hence the vector space isomorphism 
$$
\varphi:  \Tan o \aaa(\ggg;\kappa,\lambda,s)\to \ggg, \qquad a\mapsto a-o,
$$
is the required isomorphism of Lie algebras.

In the opposite direction, given a Lie affgebra $\aaa$, we need to indicate for any element $o\in \aaa$ linear functions $\kappa,\lambda$ acting on the tangent Lie algebra $\Tan o\aaa$  and an element $s\in \aaa$ which meet the condition \eqref{kl}. Comparing \eqref{aff.bra} with \eqref{tangentLie}, we expect  good candidates for these to be the following data:
    \begin{equation*}
        s=\{o,o\}\, ,
    \end{equation*}
    \begin{equation*}
        \lambda : \Tan o\aaa \rightarrow \Tan o\aaa , \quad a \mapsto \{o,a\}-\{o,o\}\, ,
        \end{equation*}
    \begin{equation*}
        \kappa: \Tan o\aaa \rightarrow \Tan o\aaa ,\quad a \mapsto\{a,a\}-\{o,o\}.
    \end{equation*}

Since $\{o,-\}$ is an affine map, it is clear that $\lambda$ is a linear transformation. On the other hand, establishing of the linearity of $\kappa$ requires one to take a closer look. First, that  $\kappa$ is an additive map follows by the bi-affine property of the Lie bracket and its affine antisymmetry \eqref{antisym}. Next, take any scalar $\alpha \in \FF$ and $a\in \aaa$ and using the bi-affine property of the Lie bracket and the base-change property Definition~\ref{def.aff}~(e) (to reduce the actions at each consecutive step to the base $o$, i.e.\ to the multiplication by scalars in $\Tan o\aaa$) compute:
$$
\begin{aligned}
    \kappa(\alpha a) &= \{\alpha \la_o a, \alpha\la_o a\} - \{o,o\} = \alpha \la _{\{o,\alpha \la _oa\}} \{a,\alpha \la _oa\} - \{o,o\}\\
    &= \alpha \{a,\alpha \la _oa\} - \alpha \{o,\alpha \la _oa\} +\{o,\alpha \la _oa\} - \{o,o\}\\
    &= \alpha\left(\alpha\la_{\{a,o\}} \{a,a\}\right) - \alpha\left(\alpha\la_{\{o,o\}} \{o,a\}\right) +\alpha\la_{\{o,o\}} \{o,a\} - \{o,o\}\\
    &= \alpha^2\left(\{a,a\} - \{a,o\} -\{o,a\} + \{o,o\}\right) + \alpha\left( \{a,o\} +\{o,a\} - \{o,o\}\right) - \alpha\{o,o\}.
\end{aligned}
$$
The required equality $\kappa(\alpha a) = \alpha\kappa(a)$ now follows by  antisymmetry of $\{-,-\}$.

In view of the affine antisymmetry the relation between the affine Lie bracket $\{-,-\}$ and its tangent Lie bracket $[-,-]$ at $o$ \eqref{tangentLie} is expressed in terms of $s$, $\lambda$ and $\kappa$ by the formula \eqref{aff.bra}, and the first part of the proof ensures that the condition $\eqref{kl}$ is fulfilled.
\end{proof}

Following \cite{LegLuk:gen} we denote by $\Delta(\ggg)$ the set of triples $(\lambda,\lambda',\lambda'')$ of linear endomorphisms of $\ggg$ that satisfy the following condition, for all $a,b\in \ggg$,
  \begin{equation}\label{gen.der}
         [\lambda(a),b] + [a,\lambda'(b)] = \lambda''\left([a,b]\right).
    \end{equation}
Any map $\lambda$ for which there exist $\lambda'$ and $\lambda''$ such that $(\lambda,\lambda',\lambda'') \in \Delta(\ggg)$ is called a {\em generalised derivation}. By slightly bending this terminology, in what follows we will refer either to $\lambda$ or the triple $(\lambda,\lambda',\lambda'')$ as to a generalised derivation. Among all generalised derivations one distinguishes several important classes. First and most standard, $\lambda$ is a {\em derivation} of $\ggg$ if and only if $(\lambda,\lambda,\lambda)\in \Delta(\ggg)$. A {\em quasi-centroid} of $\ggg$, denoted by $\mathrm{QC}(\ggg)$, is a vector space of all $\kappa \in \mathrm{Lin}(\ggg)$, such that $(\kappa,-\kappa,0)\in \Delta(\ggg)$. That is, $\kappa \in \mathrm{QC}(\ggg)$ if and only if, for all $a,b\in \ggg$,
\begin{equation}\label{quasicent}
    [\kappa(a),b] = [a,\kappa(b)].
\end{equation}
Since we assume that $\mathrm{char} \FF \neq 2$, the condition \eqref{quasicent} is equivalent to \eqref{k}, i.e.\
\begin{equation}\label{quasicent.a}
    \mathrm{QC}(\ggg) = \{\kappa \in \mathrm{Lin}(\ggg)\; |\; [\kappa(a),a] =0, \mbox{for all}\; a\in \ggg\}.
\end{equation}
It is easily seen that $\mathrm{QC(\ggg)}$ is a Lie algebra with the bracket given by the commutator (defined with respect to the composition of endomorphisms). The structure of  $\mathrm{QC(\ggg)}$ is studied in detail in \cite{LegLuk:gen}. For example,  if $Z(\ggg) = 0$, then $\mathrm{QC(\ggg)}$ is a commutative associative algebra (with respect to composition); see \cite[Theorem~5.12]{LegLuk:gen}. Next, the {\em centroid} of $\ggg$, denoted by $\mathrm{C}(\ggg)$ is the space of linear endomorphisms $\kappa$ such that $(0,\kappa,\kappa) \in \Delta(\ggg)$ or, equivalently by the antisymmetry of the Lie bracket, $(\kappa,0,\kappa) \in \Delta(\ggg)$. That is $\kappa \in \mathrm{C}(\ggg)$ if and only if, for all $a,b\in\ggg$,
\begin{equation}\label{centroid}
\kappa([a,b]) = [\kappa(a),b]= [a,\kappa(b)].
\end{equation}
Obviously $\mathrm{C}(\ggg) \subseteq \mathrm{QC}(\ggg)$ and it is clear that $\mathrm{C}(\ggg)$ is an associative unital algebra with respect to composition. Centroids of Lie algebras are studied in \cite[Chapter~X]{Jac:Lie}, where it is shown, for instance, that if $\ggg$ is a perfect Lie algebra (that is, $[\ggg,\ggg] = \ggg$) then $\mathrm{C}(\ggg)$ is a commutative algebra \cite[Chapter~X,~Lemma~1]{Jac:Lie}, while if $\ggg$ is simple, then $\mathrm{C}(\ggg)$ is a field \cite[Chapter~X,~Theorem~1]{Jac:Lie}.

The interplay between the quasicentroid and the centroid of a Lie algebra $\ggg$ is studied in \cite{LegLuk:gen}. In particular, it is proven there that  if $Z(\ggg)=0$, then $\mathrm{QC}(\ggg) = \mathrm{C}(\ggg) \oplus A$, where $A^2=0$ \cite[Theorem~5.12]{LegLuk:gen}, Furthermore, if in addition $[\ggg,\ggg]=\ggg$, then $\mathrm{QC}(\ggg) = \mathrm{C}(\ggg)$, \cite[Theorem~5.28]{LegLuk:gen}.

\begin{lemma}\label{lem.gen}

    Let $\ggg$ be a Lie algebra,  $\lambda,\delta \in \mathrm{Lin}(\ggg)$, and set $\kappa:= \lambda -\delta$. 
    \begin{zlist}
        \item The following statements are equivalent:
        \begin{rlist}
                  \item $(\delta,\lambda,\lambda)\in \Delta(\ggg)$;
        \item for all $\alpha \in \mathrm{C}(\ggg)$, $(\delta,\lambda+\alpha,\lambda+\alpha)\in \Delta(\ggg)$;
        \item for all derivations $\alpha$ of $\ggg$, $(\delta+\alpha,\lambda+\alpha,\lambda+\alpha)\in \Delta(\ggg)$;
        \item $\lambda$ and $\kappa$ satisfy \eqref{kl}.
          \end{rlist} 
        \item  If $(\delta,\lambda,\lambda)\in \Delta(\ggg)$, then:
        \begin{rlist}
         \item for all positive integers $n$, $\kappa^n\in \mathrm{QC}(\ggg)$;
    \item  for all $a\in \ggg$, the vector space $\ggg_\kappa(a)$ spanned by $\{\kappa^{n}(a)\; |\; n\in \NN\}$ is an abelian Lie subalgebra of $\ggg$;
    \item $\delta$ is a derivation if and only if $\kappa \in  \mathrm{C}(\ggg)$.
        \end{rlist}
    \end{zlist}
    \end{lemma}
\begin{proof}
(1) Since the zero map is both an element of the centroid and a derivation, that either (ii) or (iii) implies (i) is obvious. Thus assume that $(\delta,\lambda,\lambda)\in \Delta(\ggg)$. Then, for all $\alpha \in \mathrm{C}(\ggg)$ and $a,b\in \ggg$, we immediately obtain  that 
$$
[\delta(a),b] + [a,\lambda(b)] + [a,\alpha(b)]  = \lambda([a,b]) +\alpha([a,b]),
$$
as required. The implication $(i)\implies (iii)$ is equally immediate. Finally, by the antisymmetry of the Lie bracket, equation \eqref{kl} can be equivalently written as
\begin{equation}\label{gk}
    \lambda([a,b]) = [\lambda(a),b] - [\kappa(a),b]+ [a,\lambda(b)],
\end{equation}
for all $a,b\in \ggg$, and thus it is clear that it holds if and only if $(\lambda-\kappa,\lambda,\lambda)$ is a generalised derivation.

(2) In view of (1) the assumption is equivalent to the fact that $\lambda$ and $\kappa$ satisfy \eqref{gk} or \eqref{kl}.

(i) Setting $b=a$ in \eqref{gk} we immediately find that \eqref{quasicent.a} holds, so $\kappa \in \mathrm{QC}(\ggg)$. Using \eqref{quasicent.a} repeatedly, we find for all $n$ and $a\in \ggg$,
$$
[\kappa^n(a),a]=[\kappa^{n-1}(a),\kappa(a)] = \ldots = [a,\kappa^n(a)].
$$
Hence $[\kappa^n(a),a] = 0$ by the antisymmetry of the Lie bracket.

(ii) Since, for all $n\in \NN$, $\kappa^n\in \mathrm{QC}(\ggg)$, for all $k,l\in \NN$, $[\kappa^k(a),\kappa^l(a)] = 0$, as $[\kappa^k(a),\kappa^l(a)] = [\kappa^{k+l}(a),a] = 0$, and therefore $\ggg_\kappa(a)$ is an abelian Lie subalgebra of $\ggg$.

(iii) Since $\delta = \lambda - \kappa$, $\delta$ is a derivation if and only if 
$$
\lambda([a,b]) - \kappa([a,b]) = [\lambda(a),b] - [\kappa(a),b] + [a,\lambda(b)] - [a,\kappa(b)].
$$
As $\lambda$ is assumed to satisfy \eqref{gk} it immediately follows that the above equality is satisfied if and only if $\kappa\in \mathrm{C}(\ggg)$.
\end{proof}

The first assertion of Lemma~\eqref{lem.gen} allows one to state Theorem~\ref{thm.alg.aff} in the following (equivalent) form 
\begin{corollary}\label{cor.alg.ff}
    There is a one-to-one correspondence between Lie affgebras with Lie fibre $\ggg$ and generalised derivations $(\delta,\lambda,\lambda)$ of $\ggg$ supplemented with an element of $\ggg$.
\end{corollary}

We note in passing that in terms of $(\delta,\lambda,\lambda)\in \Delta(\ggg)$  and $s\in \ggg$ the bracket \eqref{aff.bra} comes out as
\begin{equation}\label{bra.der}
    \{a,b\} = [a,b] -\delta (a)+ \lambda(b) + s.
\end{equation}
Furthermore, the bracket
$$
 \{a,b\}' := [a,b] -\delta (b)+ \lambda(a) + s.
$$
 satisfies the alternative version of the affine Jacobi identity mentioned in the last paragraph of Remark~\ref{rem.new.old} (i.e.\ with the nesting of the brackets $\{\{-,-\}',-\}'$). So both versions of the Lie affgebra are obtained with exactly the same sets of data on a Lie algebra, thus confirming the claim made at the end of Remark~\ref{rem.new.old} about the conceptual equivalence of both approaches.

Theorem~\ref{thm.alg.aff} allows one to make a foray into the classification of Lie affgebras with a prescribed tangent Lie algebra (we will illustrate this in Section~\ref{sec.ex}). To make this strategy feasible however, first one needs to determine homomorphisms between Lie affgebras $\aaa(\ggg; \kappa,\lambda,s)$.

\begin{theorem}\label{thm.hom}
    A function $\varphi:\aaa(\ggg; \kappa,\lambda,s)\to \aaa(\ggg'; \kappa',\lambda',s')$ is a homomorphism of Lie affgebras if and only if there exist a Lie algebra homomorphism
$\psi: \ggg\to\ggg'$ and $q'\in \ggg'$, such that
\begin{subequations}\label{hom}
    \begin{equation}\label{hom.kk'}
        \psi \kappa = \kappa'\psi,
    \end{equation}
    \begin{equation}\label{hom.ll'}
        \psi \lambda = (\ad_{q'} +\lambda')\psi,
    \end{equation}
    \begin{equation}\label{hom.ss'}
        \psi (s) = s' - q' +\kappa'(q').
    \end{equation}
\end{subequations}
\end{theorem}
\begin{proof}
    A function $\varphi:\aaa(\ggg; \kappa,\lambda,s)\to \aaa(\ggg'; \kappa',\lambda',s')$ is a homomorphism of affine spaces if and only if the function
    $$
    \psi: \ggg \to \ggg', \qquad a\mapsto \varphi(a) - \varphi(0),
    $$
    is a linear transformation. Set $q'=\varphi(0)$.

    Written explicitly the Lie affgebra homomorphism condition, for all $a,b\in \ggg$,
    $$
    \varphi\left(\{a,b\}\right) = \{\varphi(a),\varphi(b)\},
    $$
    boils down to 
    \begin{equation}\label{hom.cond}
    \begin{aligned}
         \psi\left([a,b]\right) + \psi\kappa (a) +\psi\lambda(b-a) &+\psi(s) + q' = [\psi(a),\psi(b)]+ [q',\psi(b-a)] \\
         &+ \kappa'\psi(a) + \kappa(q')
         +\lambda'(\psi(b)-\psi(a)) + s'.
    \end{aligned}
    \end{equation}
    Setting $a=b=0$ we obtain \eqref{hom.ss'}, and thus \eqref{hom.cond} reduces to
    \begin{equation}\label{hom.cond'}
    \begin{aligned}
         \psi\left([a,b]\right) + \psi\kappa (a) &+\psi\lambda(b-a) = [\psi(a),\psi(b)] \\
         &+ [q',\psi(b-a)] 
         + \kappa'\psi(a)
         +\lambda'(\psi(b)-\psi(a)).
    \end{aligned}
    \end{equation}
    Setting $a=b$ in \eqref{hom.cond'} we obtain \eqref{hom.kk'} and thus further simplification of \eqref{hom.cond},
    \begin{equation}\label{hom.cond''}
    \begin{aligned}
         \psi\left([a,b]\right) +\psi\lambda(b-a) = [\psi(a),\psi(b)]
         + [q',\psi(b-a)] 
         +\lambda'(\psi(b)-\psi(a)).
    \end{aligned}
    \end{equation}
    Finally,  setting $a =0$ yields \eqref{hom.ll'} as well as the fact that $\psi$ is a homomorphism of Lie algebras.

    The converse is straightforward. 
\end{proof}

\begin{corollary}\label{cor.class}
    Let $\aaa = \aaa(\ggg; \kappa,\lambda,s)$ and $\aaa' = \aaa(\ggg'; \kappa',\lambda',s')$. Then $\aaa$ is isomorphic to $\aaa'$ if and only if there exist a Lie algebra isomorphism $\Psi: \ggg\to \ggg'$ and an element $q\in \ggg$ such that
    \begin{subequations}\label{iso.class}
        \begin{equation}\label{kk'}
            \kappa' = \Psi \kappa \Psi^{-1},
        \end{equation}
        \begin{equation}\label{ll'}
            \lambda' = \Psi (\lambda - \ad_q) \Psi^{-1},
        \end{equation}
        \begin{equation}\label{ss'}
            s' = \Psi(s+ q -\kappa(q)).
        \end{equation}
    \end{subequations}
\end{corollary}

\begin{proof}
    This follows immediately from Theorem~\ref{thm.hom} as the affine map $\Phi: \aaa\to \aaa'$ is an isomorphism if and only if its linear part $\Psi = \Phi - \Phi(0):\ggg\to\ggg'$ is an isomorphism of vector spaces. Rearranging equations \eqref{hom}, setting $q=\Psi^{-1}(q')$, where $q' = \Phi(0)$ and noting that $\ad_{q'} = \Psi\ad_q\Psi^{-1}$ we obtain conditions \eqref{iso.class}.
\end{proof}

We observe in passing that in terms of generalised derivations $(\delta,\lambda,\lambda)$ and $(\delta',\lambda',\lambda')$ (where $\delta = \lambda - \kappa$ and $\delta' = \lambda' - \kappa'$) equations \eqref{hom.kk'}, \eqref{kk'} take the forms
$$
 \psi \delta = (\ad_{q'} +\delta')\psi, \qquad \delta' = \Psi (\delta - \ad_q) \Psi^{-1},
$$
respectively. 

As an immediate consequence of Corollary~\ref{cor.class} we obtain the following criterion of non-isomorphism of Lie affgebras.

\begin{corollary}\label{cor.non.is}
    If Lie affgebras $\aaa$ and $\bbb$ have non-isomorphic Lie algebra fibres, then they are not isomorphic.
\end{corollary}

Theorem~\ref{thm.hom} allows one also to identify simple Lie affgebras. Following the universal algebra (or category theory) convention, we say that a Lie affgebra $\aaa$ is {\em simple} provided every homomorphism with $\aaa$ as the domain is either constant or an injective function.

\begin{corollary}\label{cor.simple}
    If $\ggg$ is a simple Lie algebra, then   $\aaa(\ggg; \kappa,\lambda,s)$ is a simple Lie affgebra. 
\end{corollary}
\begin{proof}
    Consider a Lie affgebra homomorphism $\varphi: \aaa(\ggg; \kappa,\lambda,s) \to \aaa'$. By Theorem~\ref{thm.alg.aff}, for any $o\in \aaa'$, $\aaa' = \aaa(\Tan o\aaa'; \kappa',\lambda',s')$. In view of Theorem~\ref{thm.hom}, $\varphi = \psi +q'$, where $\psi:\ggg\to \Tan o\aaa'$ is a Lie algebra homomorphism satisfying conditions \eqref{hom}. Since $\ggg$ is a simple Lie algebra $\psi$ is either the zero map, in which case $\varphi$ is constant or injective, in which case also $\varphi$ is injective.
\end{proof}

Given a vector space $V$ and its subspace $U$, any coset $v+U\subseteq V$ is an affine subspace of $V$ in a natural way:
$$
\langle v+u, v+u',v+u''\rangle = v+ (u-u'+u''), \quad \alpha \la_{v+u} (v+u') = v+ (1-\alpha u) +\alpha u',
$$
for all $u,u',u'' \in U$. The criterion when a coset of a subspace of a Lie algebra $\ggg$ yields a Lie subaffgebra of $\aaa (\ggg;\kappa,\lambda,s)$ is given in the following

\begin{proposition}\label{prop.sub.Lie}
     Let $\aaa = \aaa (\ggg;\kappa,\lambda,s)$ be a Lie affgebra and let $\bbb$ be an affine subspace of $\aaa$. Then $\bbb$ is a Lie subaffgebra of $\aaa$ if and only if there exist $a\in \bbb$ and a Lie subalgebra $\hhh$ of $\aaa$, such that $\bbb = a +\hhh$ and 
  \begin{blist}
      \item $\kappa(a)-a+s \in \hhh$;
      \item $\kappa(\hhh) \subseteq \hhh$;
      \item $(\lambda+\ad_a)(\hhh)\subseteq \hhh$.
  \end{blist}

Furthermore, 
$$\bbb \cong \aaa(\hhh;\kappa, \lambda+\ad_a, \kappa(a)-a+s),$$
as Lie affgebras.
\end{proposition}
\begin{proof}
    An affine subspace $\bbb$ is necessarily of the form $a+\hhh$, where $\hhh$ is a vector subspace of $\ggg$. Then  $\bbb = a+\hhh$ is a Lie subaffgebra of $\aaa$ if and only if, for all $v,w\in V$,
    \begin{equation}\label{sub.Lie}
        \{a+v,a+w\} = [v,w] + [a,w] +[v,a] +\kappa(a) +\kappa(v) +\lambda(w-v) +s \in a+ \hhh.
    \end{equation}
    Setting $v=w=0$ in \eqref{sub.Lie} we obtain property (a). In view of (a), setting  $v=w$ yields (b), while setting $w=0$ yields (c). All this then necessarily implies that $[v,w]\in\hhh$, hence $\hhh$ is a Lie subalgebra of $\ggg$.  Conversely, if $\hhh$ is a Lie subalgebra of $\ggg$ and (a)--(c) hold, $\{a+v,a+w\}\in a+\hhh$, so $a+\hhh$ is a Lie subaffgebra of $\aaa$. 

    Properties (a)--(c) imply that, when restricted to $\hhh$, $\kappa, \lambda+\ad_a \in \mathrm{Lin}(\hhh)$, while $\kappa(a)-a+s\in \hhh$. Since $\lambda$ satisfies \eqref{kl}, its restriction to $\hhh$ shifted by the inner derivation $\ad_a$ also satisfies property \eqref{kl} by Lemma~\ref{lem.gen}. Hence $\aaa(\hhh;\kappa, \lambda+\ad_a, \kappa(a)-a+s)$ is a Lie affgebra. Its isomorphism with $a+\hhh$ is given by $v\mapsto a+v$.
\end{proof}

\section{Derivation-type Lie affgebras and Lie algebra hulls}\label{sec.idem}
As explained in \cite[Proposition~3.5]{BrzPap:Lie} there is a bijective correspondence between idempotent Lie affgebra brackets on an affine space $\aaa$ and Lie brackets valued in the tangent vector space at any given element of $\aaa$ that satisfy Lie affgebra axioms as stated in \cite{GraGra:Lie}. On the other hand \cite[Theorem~10]{GraGra:Lie} establishes that the latter can be lifted up to a Lie bracket on a vector space hull of $\aaa$, that is, a universally defined vector space in which $\aaa$ is a hyperplane of codimension one. The resulting Lie algebra is called the {\em Lie algebra hull} of $\aaa$. In this section we extend Lie algebra hulls to a wider class of Lie affgebras and interpret them as extensions of a tangent Lie algebra by derivation.

\begin{definition}\label{def.deriv.type}
    A Lie affgebra $\aaa$ is said to be {\em derivation-type} if $\aaa\cong\aaa(\ggg; \kappa, \lambda,s)$ such that $\delta:= \lambda-\kappa $ is a derivation of $\ggg$.
\end{definition}

Lemma~\ref{lem.gen} immediately implies that
    a Lie affgebra $\aaa(\ggg; \kappa, \lambda,s)$ is a derivation-type Lie affgebra if and only if $\kappa\in \mathrm{C}(\ggg)$.
Thus the study of derivation-type Lie affgebras in large part boils down to the study of centroids of Lie algebras

\begin{proposition}\label{prop.perfect.simple}~
    \begin{zlist}
    \item Any finite dimensional Lie affgebra with a centreless perfect Lie algebra fibre is a derivation-type Lie affgebra.
        \item If $\FF$ is an algebraically closed field and $\ggg$ is a simple Lie algebra, then:
        \begin{rlist}
            \item Any Lie affgebra with $\ggg$ as the tangent Lie algebra is necessarily isomorphic to $\aaa(\ggg; \kappa\,\id, \kappa\,\id, s)$, for some $s\in \ggg$ and $\kappa\in \FF$.  In particular, the affine Lie bracket is
    \begin{equation}\label{bra.semi}
        \{a,b\} = [a,b] + \kappa\, b +s,
    \end{equation}
    for all $a,b\in \ggg$.
    \item Furthermore,
    \begin{equation}\label{isomorphisms}
    \aaa(\ggg;\kappa\,\id,\kappa\,\id, s)\cong \aaa(\ggg;\kappa'\,\id,\kappa'\,\id, s'),
\end{equation} 
as Lie affgebras if and only if $\kappa'=\kappa$ and $s' = \Psi(s)$, where $\Psi$ is an automorphism of $\ggg$.
        \end{rlist}
    \end{zlist} 
\end{proposition}
\begin{proof}
    
    (1) The statement is a straightforward consequence of \cite[Theorem~5.28]{LegLuk:gen}, which states that if $Z(\ggg)=0$ and $[\ggg,\ggg]=\ggg$, then $\mathrm{QC}(\ggg) = \mathrm{C}(\ggg)$.

(2)(i) Since a simple Lie algebra satisfies the hypothesis of the first statement, $\aaa(\ggg; \kappa, \lambda, s)$ is a derivation-type Lie affgebra. All derivations of a simple Lie algebra are inner, hence
    $$
    \lambda - \kappa = \ad_q,
    $$
    for some $q\in \ggg$. Choosing this $q$ in Corollary~\ref{cor.class} and using \eqref{ll'}, we obtain $\lambda'=\kappa'$, as required. The formula for the bracket follows immediately from \eqref{aff.bra}. Furthermore, since $\kappa \in \mathrm{C}(\ggg)$, necessarily $\kappa \circ \ad_a = \ad_a \circ \kappa$, for all $a\in \ggg$ and by the simplicity of $\ggg$  and  Schur's lemma, $\kappa$ is a scalar multiple of the identity (compare \cite[Chapter~X, Theorem ~1]{Jac:Lie} or \cite[Corollary~5.5]{BurDek:pos}).

    (2)(ii) Since $\ggg$ has the trivial centre, the constant $q$ in Corollary~\ref{cor.class} is necessary zero, as any other choice would separate $\lambda$ from $\kappa$. Thus isomorphisms \eqref{isomorphisms} correspond to Lie algebra automorphisms $\Psi$ of $\ggg$, which in  view of Corollary~\ref{cor.class} do not modify $\kappa$ while changing $s$ to $s' = \Psi(s)$.
\end{proof}

The automorphisms of simple complex Lie algebras are well-known; see e.g.\ \cite[Chapter~IX]{Jac:Lie}. In particular, in the case of series $A_n$, i.e.,  Lie algebras of traceless matrices $\sss\lll(n+1,\CC)$, the isomorphisms \eqref{isomorphisms} take place if and only if there is a special linear matrix $A\in \mathrm{SL}(n,\CC)$ such that $s' = A^{-1}sA$ in all cases or $s' = -A^{-1}s^{T}A$ if $n\geq 2$; see \cite[Chapter~IX, Theorem~5]{Jac:Lie}. In the case of the Lie algebra of odd-dimensional antisymmetric matrices (the $B_n$ series) of dimension at least 5 (i.e., $n\geq 2$) and even-dimensional antisymmetric matrices (the $D_n$ series) of dimension at least 10 (i.e., $n\geq 5$), the isomorphisms \eqref{isomorphisms} hold if and only if there is an orthogonal matrix $O$ such that $s' = O^{-1}sO$; see \cite[Chapter~IX, Theorem~6]{Jac:Lie}. Finally, in the case of the  Lie algebras from the series $C_n$ with $n\geq 3$, the correspondence \eqref{isomorphisms} is given by symplectic matrices $M$, $s' = M^{-1}sM$; see again \cite[Chapter~IX, Theorem~6]{Jac:Lie}. In Section~\ref{sec.ex.sl2} we give an explicit description of isomorphism classes of Lie affgebras with the fibres $\sss\lll(2,\CC)$.

Recall that if $\delta$ is a derivation of a Lie algebra $\ggg$, then one can define the semi-direct product Lie algebra  $ \ggg(\delta) := \{x + \alpha \delta\; |\; x\in \ggg, \alpha \in \FF\}\cong \ggg \oplus \FF$ with the bracket
\begin{equation}\label{Lie.ext}
    [x+\alpha\delta, y+\beta\delta] = [x,y] + \alpha \delta(y) - \beta\delta(x),
\end{equation}
for all $x,y\in \ggg$ and scalars $\alpha,\beta\in \FF$. It is clear that $\ggg$ is an ideal in  the Lie algebra $\ggg(\delta)$.

\begin{proposition}\label{prop.hull}
    Let $\aaa$ be a derivation-type Lie affgebra isomorphic to  $\aaa(\ggg; \kappa\,\id, \lambda, s)$, for some scalar $\kappa \in \FF$. Set $\delta = \lambda - \kappa\,\id$. Then $\aaa \cong \ggg +\delta$, where  $\ggg +\delta$ is the Lie subaffgebra of $\aaa(\ggg(\delta); \kappa\,\id,\kappa\,\id, s+ (1-\kappa)\delta) $.
\end{proposition}
\begin{proof}
    The Lie bracket in $\aaa(\ggg(\delta); \kappa\,\id,\kappa\,\id, s+ (1-\kappa)\delta) $ comes out as
    $$
     \{a+\alpha\delta, b+\beta\delta\} = [a,b] + \alpha \delta(b) - \beta\delta(a) +\kappa b + s + (\kappa(\beta -1)+1)\delta.
    $$
    Since the Lie bracket in $\aaa(\ggg; \kappa\,\id, \lambda, s)$ is
    $$
    \{a,b\} = [a,b] +\kappa b +\delta(b-a) +s,
    $$
    we find that
    $$
    \{a+\alpha\delta, b+\beta\delta\} = \{a,b\} + (\alpha -1) \delta(b) + (1-\beta)\delta(a) + (\kappa(\beta-1)+1)\delta.
    $$
    In particular $\{a+\delta,b+\delta\} = \{a,b\} +\delta$ and the required isomorphism of Lie affgebras is $a\mapsto a +\delta$.
\end{proof}

\begin{remark}\label{rem.simple}
    In view of Proposition~\ref{prop.perfect.simple}, Proposition~\ref{prop.hull} applies to all Lie affgebras with simple fibres. Proposition~\ref{prop.perfect.simple} implies further that, over an algebraically closed field, up to isomorphism, $\lambda = \kappa$ and hence $\delta =0$. Thus, over an algebraically closed field, if $\aaa$ is fibred by a simple Lie algebra $\ggg$, then the Lie hull of $\aaa$ is isomorphic to the trivial extension of $\ggg$ by $\delta$ so that
    $$
    [a,\delta] =0,
    $$
    for all $a\in \ggg$.
\end{remark} 

\begin{proposition}\label{prop.noab}
    Let $\aaa$ be a Lie affgebra fibred by a Lie algebra $\ggg$ which does not contain a non-trivial abelian subalgebra. Then $\aaa$ is a derivation-type Lie affgebra satisfying the hypothesis of Proposition~\ref{prop.hull}.
\end{proposition}
\begin{proof}
 Let $\aaa\cong \aaa(\ggg; \kappa, \lambda, s)$. Take any basis $\{a_i\}$ of $\aaa$ and consider vector spaces $\ggg_\kappa(a_i)$ spanned by  $\{\kappa^n(a_i)\; | \; n\in \NN\}$. In view of Lemma~\ref{lem.gen}, each $\ggg_\kappa(a_i)$ is an abelian subalgebra of $\ggg$. Therefore, by hypothesis, $\ggg_\kappa(a_i)$ are one-dimensional spaces, and there exist $\kappa_i \in \FF$, such that $\kappa(a_i) = \kappa_i a_i$. Next, since $\kappa \in \mathrm{QC}(\ggg)$,  $\kappa_i[a_i,a_j] = \kappa_j[a_i,a_j]$. If $i\neq j$, then again by the hypothesis $[a_i,a_j]\neq 0$, and so $\kappa_i=\kappa_j$, for all $i,j$. This proves that $\kappa$ is necessarily a multiple of the identity, and the statement follows.
\end{proof}
\begin{corollary}\label{cor.hull}
    Let $\aaa$ be a Lie affgebra. The following statements are equivalent:
    \begin{zlist}
        \item The bracket $\{-,-\}$ is an idempotent operation, i.e., $\{a,a\} =a$, for all $a\in\aaa$.
     \item There exist a Lie algebra $\ggg$ and a derivation $\delta$ of $\ggg$ such that $\aaa \cong \aaa (\ggg; \id,  \delta+\id, 0)$.
        \item For any $o\in \aaa$, there is a derivation $\delta$  of the tangent Lie algebra $\Tan o \aaa$ such that $\aaa \cong \Tan o \aaa + \delta$, where $\Tan o \aaa + \delta$ is a Lie subaffgebra of a Lie affgebra $\aaa(\delta)\cong \aaa\times \FF$ with the bracket 
        $$
        \{x, y\} = [x, y] + y,
        $$ 
        for all $x,y\in (\Tan o \aaa) (\delta)$.
    \end{zlist}
\end{corollary}
    
\begin{proof}
    By Theorem~\ref{thm.alg.aff}, $\aaa \cong \aaa(\ggg;\kappa,\lambda,s)$. In view of \eqref{aff.bra}, $\{a,a\} = a$, for all $a\in \aaa$ if and only if
    $$
    a = \kappa(a) +s.
    $$
    Since $\kappa$ is a linear endomorphism, $s=0$ and $\kappa(a) = a$, and hence $\aaa$ is a derivation-type Lie affgebra with derivation $\delta = \lambda -\id $. This establishes the equivalence of (1) and (2). 

    Assuming (2),  by Theorem~\ref{thm.alg.aff}, given any $o\in \aaa$ we can choose $\ggg = \Tan o \aaa$. Then $\aaa(\delta) = \aaa(\ggg;\id,\id,0)$ and (3) follows by Proposition~\ref{prop.hull}. The implication $(3)\implies (1)$ is immediate.
\end{proof}

\begin{example}\label{ex.act.hull}
    The Lie bracket \eqref{action.bracket} of the Lie affgebra $\aaa_{\zeta,\id}$ described in Example~\ref{ex.Lie.Lie} is an idempotent operation by \eqref{idem}, and hence Corollary~\ref{cor.hull} can be applied. The map $\lambda$ comes out as $\lambda(a) = \zeta a$, so that $\delta(a) = (\zeta-1)a$ and the Lie bracket of the hull of $\aaa_{\zeta,\id}$ reads
    $$
    [a+\alpha \delta, b+\beta \delta] = (\zeta -1)(\alpha b - \beta a),
    $$
    for all $a,b\in \aaa$ and $\alpha, \beta \in \FF$.
\end{example}
 \section{Examples}\label{sec.ex}

\subsection{One-dimensional Lie affgebras}\label{sec.one.dim}

All one-dimensional Lie affgebras must have the unique one-dimensional abelian Lie algebra as fibres. Let $e$ be a fixed basis of $\ggg$. There are no restrictions on $\lambda$ and $\kappa$, which are both scalar multiples of identity, say $\lambda(e) = \lambda e$ and $\kappa(e) = \kappa e$, $\kappa,\lambda \in \FF$, and are not affected by an automorphism $\Psi$ of $\ggg$, which itself is a non-zero scalar multiple of identity, say $\Psi(e) = \psi e$. The adjoint action is trivial. Let $s=\sigma e$ and $q=\xi e$. Then $s$ can be transferred to
$$
s' = \psi(\sigma + (1-\kappa)\xi)e.
$$
If $\kappa\neq 1$, $\xi$ can be chosen so that $s'=0$. If $\kappa =1$ then either $\sigma =0$ or it can be normalised to $1$ by choosing $\psi$. Therefore, there are two families of non-isomorphic Lie affgebra structures on a one dimensional (affine) space with corresponding brackets
$$
 \{a,b\}  =\begin{cases}
    ( \kappa-\lambda)a + \lambda b & \mbox{or}\\
    (1 - \lambda)a + \lambda b +e.
\end{cases}
$$
All these are derviation-type Lie affgebras which fulfil the hypothesis of Proposition~\ref{prop.hull}. Since the derivations are $\delta(e) = (\lambda - \kappa)e,$ the Lie hulls are spanned by $e$ and $\delta$ with the bracket
$$
[e,\delta] = (\kappa - \lambda)e,
$$
and thus exhaust all non-isomorphic classes of two-dimensional Lie algebras: either $[e,\delta] = 0$ or $[e,\delta] =e$.

\subsection{Lie affgebras with tangent abelian Lie algebras}\label{sec.ex.ab}
 This is more of a non-example rather than an example, but we include this discussion here as an illustration of the richness of the variety of Lie affgebras fibred by the same Lie algebra. We take $\ggg$ to be an $n$-dimensional vector space with the trivial bracket $[a,b]=0$. The conditions \eqref{kl} and \eqref{centroid} are empty, so any linear endomorphisms $\kappa$ and $\lambda$ and a vector $s\in \ggg$ give a derivation-type Lie affgebra $\aaa(\ggg;\kappa,\lambda,s)$, with the bracket
 \begin{equation}\label{ab}
  \{a,b\} = \kappa(a) +\lambda(b-a) +s,  
 \end{equation}
  for all $a,b\in \ggg$.

 Any linear automorphism $\Psi$ of $\ggg$ is a Lie algebra homomorphism, the adjoint action is trivial and hence the equivalence classes of Lie affgebras are controlled by the general similarity relation between  $\kappa$ and $\kappa'$ or $\lambda$ and $\lambda'$, and on the multiplicity of eigenvalue 1 of $\kappa$ or, equivalently, the rank of $\id - \kappa$. If this rank  is equal to $n$, then $q$ can be chosen so that $s'=0$ and thus the constant term from \eqref{ab} can be removed. If the rank is equal to $r$, then the constant term will depend on $n-r$ free parameters.

 \subsection{Two dimensional Lie affgebras with non-abelian fibres}\label{sec.ex.xy}

 Let $\bbb$ be the Borel subalgebra of $\sss\lll(2,\CC)$, that is, $\bbb$ is the 2-dimensional Lie algebra with a basis $E: e_1,e_2$ and the bracket $[e_1,e_2] = e_1$. By Proposition~\ref{prop.noab} $\kappa$ is necessarily a scalar multiple of identity, so that $\kappa\in \mathrm{C}(\bbb)$. Therefore, $\delta =\lambda - \kappa$ is a derivation by Lemma~\ref{lem.gen}. One easily checks that all derivations of $\bbb$ are inner, and hence $\lambda = \kappa + \ad_q$, where  $q\in \bbb$. In view of this and equation \eqref{ll'}  in Corollary~\ref{cor.class}, any Lie affgebra over $\bbb$ is isomorphic to one with $\lambda = \kappa = \gamma\, \id$, where $\gamma \in \CC$. 
 
 Since $\kappa$ and $\lambda$ are scalar multiples of identity, they are preserved by any automorphism $\Psi$ of $\bbb$. One easily finds that, in the basis $E$,
 $$
\Psi(a,b) =  \begin{pmatrix}
    a & b \cr 0 & 1
\end{pmatrix},
$$
where $a,b\in \CC$, $a\neq 0$. With no loss of generality we can choose $q=0$ in \eqref{ss'}, so that $s' = \Psi(s)$, or in the basis $E$,
$$
\begin{pmatrix}
    s'_1  \cr  s'_2
\end{pmatrix} = \begin{pmatrix}
    a & b \cr 0 & 1
\end{pmatrix} \begin{pmatrix}
    s_1  \cr  s_2
\end{pmatrix}
=
\begin{pmatrix}
       as_1+bs_2 \cr s_2 
    \end{pmatrix}.
    $$
If $s_2\neq 0$, then $b$ can be chosen so that $s'_1=0$. Otherwise, either $s_1=0$ and hence $s'=0$ or $a$ can be chosen in such a way that $s'_1 =1$. Put together we obtain the following two families of  non-isomorphic Lie affgebras
 $$
 \begin{aligned}
    \aaa(\bbb; \gamma\, \id,\gamma\, \id, \sigma e_2) \; &: \quad \{x,y\} = [x,y] +\gamma y +\sigma e_2, \\
    \aaa(\bbb; \gamma\, \id,\gamma\, \id, e_1) \; &: \quad \{x,y\} = [x,y] +\gamma y +e_1 ,
 \end{aligned}
  $$
 where $\gamma,\sigma \in \CC$. As up to isomorphism $\bbb$ is the only non-abelian two-dimensional Lie algebra, this classifies (up to isomorphism) all two-dimensional Lie affgebras with a non-abelian tangent Lie algebra.

 By Proposition~\ref{prop.noab}, all these affgebras are derivation-type affgebras that fulfil hypothesis of Proposition~\ref{prop.hull}, and since $\lambda =\kappa$ the derivation $\lambda-\kappa$ is trivial. Consequently, the Lie algebra hull is spanned by $e_1,e_2,\delta$, with the bracket
 $$
 [e_1,e_2]= e_1, \qquad [\delta,e_1]= 0, \qquad [\delta,e_2] =0.
 $$
 
  \subsection{Lie affgebras fibred by $\sss\lll(2,\CC)$}\label{sec.ex.sl2} 
  Let us write the $\mathfrak{sl}(2,\CC)$ algebra in the Chevalley basis $C: e,h,f$,
  \begin{equation}\label{sl2.rel}
      [e,f] = h, \quad [h,e] = 2e, \quad [h,f]=-2f.
  \end{equation}
  Since $\sss\lll(2,\CC)$ is a simple Lie algebra over an algebraically closed field, by Proposition~\ref{prop.perfect.simple} every Lie affgebra with $\sss\lll(2,\CC)$ as a fibre is isomorphic to one with the bracket
$$
\{x,y\} = [x,y] +\gamma y +s, \qquad \gamma\in \CC, s\in \sss\lll(2,\CC).
$$
Equivalence classes corresponding to different choices of $s$ are controlled by the automorphisms of $\sss\lll(2,\CC)$ through \eqref{ss'}. Any such automorphism is given by the conjugation with an element of the Lie group $\mathrm{SL}(2,\CC)$, and thus, 
$$
\begin{pmatrix}
    s'_{2} & s'_{1}\cr s'_{3} & -s'_{2}
\end{pmatrix}\sim \begin{pmatrix}
    s_{2} & s_{1}\cr s_{3} & -s_{2}
\end{pmatrix},
$$
 provided there exists $u\in \mathrm{SL}(2,\CC)$ such that
$$
\begin{pmatrix}
    s'_{2} & s'_{1}\cr s'_{3} & -s'_{2}
\end{pmatrix} = u^{-1} \begin{pmatrix}
    s_{2} & s_{1}\cr s_{3} & -s_{2}
\end{pmatrix} u.
$$
First consider,
$$
\begin{pmatrix}
    1 & a\cr 0 & 1
\end{pmatrix}^{-1}\begin{pmatrix}
    s_{2} & s_{1}\cr s_{3} & -s_{2}
\end{pmatrix}\begin{pmatrix}
    1 & a\cr 0 & 1
\end{pmatrix} = \begin{pmatrix}
    s_2-as_3 & 2as_2 + s_1 -a^2 s_3 \cr s_3 & -s_2+as_3
\end{pmatrix}.
$$
If $s_3\neq 0$ diagonal terms can be eliminated by a suitable choice of $a$. If $s_3=0$ we can obtain a diagonal matrix (as long as $s_2\neq 0$) or a strictly upper triangular one. Put together we obtain that
$$
s'= \begin{pmatrix}
    \sigma  & 0\cr 0 & -\sigma
\end{pmatrix} = \sigma\, h \quad \mbox{or} \quad s' = \begin{pmatrix}
    0  & \sigma_1 \cr \sigma_2 & 0
\end{pmatrix}.
$$
The second case can be reduced further by applying
$$
\begin{pmatrix}
    b & 0\cr 0 & b^{-1}
\end{pmatrix}^{-1}\begin{pmatrix}
    0  & \sigma_1 \cr \sigma_2 & 0
\end{pmatrix}\begin{pmatrix}
    b & 0\cr 0 & b^{-1}
\end{pmatrix} = \begin{pmatrix}
    0  & b^{-2}\sigma_1 \cr b^2 \sigma_2 & 0
\end{pmatrix}.
$$
Depending on the nullity of the $\sigma_i$ in addition to the zero matrix we obtain three possibilities
$$
\begin{pmatrix}
    0 & 1\cr 0 & 0
\end{pmatrix} = e, \qquad \begin{pmatrix}
    0 & 0\cr 1 & 0
\end{pmatrix}  =f, \qquad \begin{pmatrix}
    0 & 1\cr \sigma^2 & 0
\end{pmatrix}, \qquad \sigma\neq 0.
$$
The last matrix belongs to the diagonal class, since
$$
\begin{pmatrix}
    \sigma & 0 \cr 0 & -\sigma
\end{pmatrix} = \begin{pmatrix}
    1 & -\frac{1}{2\sigma} \cr \sigma & \frac{1}{2}
\end{pmatrix}^{-1}\begin{pmatrix}
0 & 1 \cr \sigma^2  & 0
\end{pmatrix}\begin{pmatrix}
    1 & -\frac{1}{2\sigma} \cr \sigma & \frac{1}{2}
\end{pmatrix}.
$$
Therefore isomorphism classes of Lie affgebras with the tangent Lie algebra $\sss\lll(2,\CC)$ split into three cases
$$
\begin{aligned}
    \aaa (\sss\lll(2,\CC); \gamma\, \id, \gamma\, \id, e)\, : &  \quad  \{x,y\} = [x,y] +\gamma y + e, \cr 
    \aaa (\sss\lll(2,\CC); \gamma\, \id, \gamma\, \id, \sigma h) \, : & \quad \{x,y\} = [x,y] +\gamma y +  \sigma h, \cr
    \aaa (\sss\lll(2,\CC); \gamma\, \id, \gamma\, \id, f) \, : & \quad \{x,y\} = [x,y] +\gamma y +  f,
    \end{aligned}
$$
where $\gamma, \sigma \in \CC$.

 \subsection{Lie affgebras fibred by $\sss\ooo(3)$.}\label{sec.ex.so3} Let $x_1,x_2,x_3$ be the standard basis of the real Lie algebra $\sss\ooo(3)$, so that,
 $$
 [x_i,x_j] = \varepsilon_{ijk}x_k,
 $$
 where $\varepsilon_{ijk}$ is the antisymmetric Levi-Civita symbol. 
 Since $\sss\ooo(3)$ has no non-trivial abelian subalgebras,
$\kappa$ is a multiple of identity by Proposition~\ref{prop.noab}. Since $\sss\ooo(3)$ is a simple  Lie algebra, all derivations are inner, which allows one to choose $\lambda = \kappa = \gamma \id$, where $\gamma \in \RR$. 

The automorphisms of $\sss\ooo(3)$ are given by the conjugation by elements of the rotation group $SO(3)$. By suitably rotating the axes, one can transform a general element of $\sss\ooo(3)$ to  $\sigma x_1$, with the uniquely determined $\sigma \in \RR$. Thus any real Lie affgebra fibred by $\sss\ooo(3)$ is isomorphic to one with the bracket, for all $x,y\in \sss\ooo(3)$,
$$
\{x,y\} = [x,y] +\gamma y + \sigma x_1,
$$
where $\gamma,\sigma \in \RR$. 

\subsection{Normalised affine matrices}\label{sec.ex.gna}
The Lie affgebras of general and special normalised affine matrices were introduced in \cite{Brz:spe} as affine subspaces of $\ggg\lll(n+1,\FF)$ and $\sss\lll (n+1,\FF)$, respectively given by
$$
\begin{aligned}
    \ggg\nnn\aaa(n,\FF) &:= \{X\in \ggg\lll(n+1,\FF)\; |\; XE=EX=E\}, \\  \sss\nnn\aaa(n,\FF) &:= \{X\in \sss\lll(n+1,\FF)\; |\; XE=EX=E\},
\end{aligned}
$$
where $E$ is the matrix in which all entries are equal to 1, that is
(the sum of all $(n+1)\times(n+1)$ matrix units $E_{ij}$). In other words $\ggg\nnn\aaa(n,\FF)$ and $\sss\nnn\aaa(n,\FF)$ contain matrices normalised so that the sum of all entries in each row and column is equal to 1. It has been shown in \cite{BrzRad:mat} that, in the terminology of the present text, 
$$
\ggg\nnn\aaa(n,\FF) \cong \aaa(\ggg\lll_0(n,\FF); \id,\id, 0), \qquad \sss\nnn\aaa(n,\FF) \cong \aaa(\sss\lll_0(n,\FF); \id,\id, 0),
$$
where the Lie algebras are defined as subalgebras of $\ggg\lll(n+1,\FF)$ and  $\sss\lll (n+1,\FF)$, respectively, by
$$
\begin{aligned}
    \ggg\lll_0(n,\FF) &:= \{X\in \ggg\lll(n+1,\FF)\; |\; XE=EX=0\}, \\  \sss\lll_0(n,\FF) &:= \{X\in \sss\lll(n+1,\FF)\; |\; XE=EX=0\}.
\end{aligned}
$$
Furthermore, if the characteristic of $\FF$ does not divide $n$ or $n+1$, $\ggg\nnn\aaa(n,\FF)$ and $\sss\nnn\aaa(n,\FF)$ fit into  Proposition~\ref{prop.sub.Lie} with $\hhh \cong \ggg\lll_0(n,\FF)\cong \ggg\lll(n, \FF)$ and $\hhh \cong \sss\lll_0(n,\FF)\cong\sss\lll(n, \FF)$, respectively. Thus
$$
\ggg\nnn\aaa(n,\FF) \cong \aaa(\ggg\lll(n, \FF);\id,\id,0), \qquad \sss\nnn\aaa(n,\FF) \cong \aaa(\sss\lll(n, \FF);\id,\id,0).
$$
The proof of these isomorphisms given in \cite[Theorem~3.2]{BrzRad:mat} is based on constructing explicit similarity transformations whose existence relies on the fact that $n+1$ is not a multiple of the characteristic of $\FF$. Here we show that indeed this assumption is necessary. 

Assume that  $\mathrm{char}\FF \mid  n+1$. Obviously, $\ggg\nnn\aaa(n,\FF) = I_{n+1} + \ggg\lll_0(n,\FF)$, where $I_{n+1}$ is the identity matrix, and the conditions of Proposition~\ref{prop.sub.Lie} are satisfied with $\hhh \cong \ggg\lll_0(n,\FF)$. Since $\tr I_{n+1} =0$, $I_{n+1}\in \sss\nnn\aaa(n,\FF)$, and so $\sss\nnn\aaa(n,\FF) = I_{n+1} + \sss\lll_0(n,\FF)$. Furthermore, $E\in \sss\lll_0(n,\FF)$ is a central element. On the other hand $\sss\lll(n,\FF)$ has a trivial centre, and so $\sss\lll_0(n,\FF)\not\cong\sss\lll(n,\FF)$ as Lie algebras. By Corollary~\ref{cor.non.is}, 
\begin{equation}\label{not.sln}
    \sss\nnn\aaa(n,\FF) \not\cong \aaa(\sss\lll(n, \FF);\id,\id,0).
\end{equation}
Next, consider the matrices
$$
A_{i\,j} = E_{i\,j} -E_{i\,n+1} - E_{n+1\, j} + E_{n+1\, n+1},
$$
$i,j = 1,\ldots n$, which form a basis for $\ggg\lll_0(n,\FF)$. Since the characteristic of $\FF$ is a factor of $n+1$, 
$
E = \sum_{i,j} A_{i\, j},
$
and
$$
\sum_{i\neq j} [ A_{i\, i}, A_{i\, j}] = \sum_{i\neq j}\left(2A_{i\,j} - A_{i\,i}\right) = 2\sum_{i,j} A_{i\,j} =2E.
$$
Thus, if $\mathrm{char}\,\FF \neq 2$, $E\in [\ggg\lll_0(n,\FF),\ggg\lll_0(n,\FF)]$. In the case of the even characteristic one finds that
$$
[\ggg\lll_0(n,\FF),\ggg\lll_0(n,\FF)]\ni \sum_{i<j}[A_{i\,i},A_{j\,j}] = \sum_{i<j}\left(A_{i\,j}+ A_{j\,i}\right) = E - I_{n+1}.
$$
Therefore, irrespective of the parity of $\mathrm{char}\,\FF$,  $[\ggg\lll_0(n,\FF),\ggg\lll_0(n,\FF)]$ has non-trivial intersection with the centre of $\ggg\lll_0(n,\FF)$. Since there are no non-trivial central elements of $\ggg\lll(n,\FF)$ that are also in $[\ggg\lll(n,\FF),\ggg\lll(n,\FF)]$, we conclude that
\begin{equation}\label{not.gln}
\ggg\nnn\aaa(n,\FF) \not\cong \aaa(\ggg\lll(n, \FF);\id,\id,0). 
\end{equation}

We now look at $\sss\nnn\aaa(n,\FF)$ in case $\mathrm{char}\FF \mid  n$. Let 
$$
A = E_{12}+E_{23}+\ldots + E_{n\,n+1}+ E_{n+1\,1}\in \sss\nnn\aaa(n,\FF).
$$
Since $AE-EA =0$, for all $X\in \sss\lll_0(n,\FF)$, $AX-XA\in \sss\lll_0(n,\FF)$, and hence 
$$
\sss\nnn\aaa(n,\FF) = A + \sss\lll_0(n,\FF) \cong \aaa(\sss\lll_0(n,\FF); \id,\id+\ad_A, 0),
$$
by Proposition~\ref{prop.sub.Lie}. 
Following \cite{BrzRad:mat} we can introduce the 
$(n+1)\times (n+1)$-matrix
     \begin{equation}\label{mat.P}
         {P} = 
         \begin{pmatrix}
             1 & 1 & 1 &\cdots & 1 & 1 \cr
             0 & 0 & 0 &\cdots & -1 & 1 \cr
             \cdots  & \cdots  & \cdots &\cdots & \cdots & \cdots \cr
             0 & 0 & -1 &\cdots & 0 & 1 \cr
             0 & -1 & 0 &\cdots & 0 & 1 \cr
             -1 & 0 & 0 &\cdots & 0 & 1 \cr
         \end{pmatrix}.
     \end{equation}
Since  $\mathrm{char}(\FF)\nmid n+1$, ${P}$ is invertible with the inverse
     $$
    {P}^{-1} =  \begin{pmatrix}
             1 & 1 & 1 &\cdots & 1 & 0  \cr
             1 & 1 & 1 &\cdots & 0 & 1  \cr
             \cdots  & \cdots  & \cdots &\cdots & \cdots & \cdots  \cr
             1 & 1 & 0 &\cdots & 1 & 1  \cr
             1 & 0 & 1 &\cdots & 1 & 1 \cr
             1 & 1 & 1& \cdots &1 & 1
         \end{pmatrix}.
        $$

As shown in \cite{BrzRad:mat}, for all $X\in \sss\lll_0(n,\FF)$, 
$$
P^{-1}XP \in \begin{pmatrix} \sss\lll(n,\FF) & 0\cr 0 & 0
    \end{pmatrix}.
$$
Furthermore,
$$
B := P^{-1}AP = \begin{pmatrix}
    -1 & -1 & -1 & \cdots &-1 & -1 & 0 \cr
    1 & 0 & 0 & \cdots & 0& 0 & 0 \cr
    0 & 1 & 0 & \cdots & 0 & 0 & 0 \cr
    \cdots & \cdots & \cdots & \cdots & \cdots & \cdots  & \cdots \cr
     0 & 0 & 0 & \cdots &1 & 0 & 0 \cr
     0 & 0 & 0 & \cdots &0 & 0 & 1
\end{pmatrix}.
$$
Using the similarity transformation given by $P$ we thus obtain,
$$
\sss\nnn\aaa(n,\FF) \cong  B + \begin{pmatrix} \sss\lll(n,\FF) & 0\cr 0 & 0
    \end{pmatrix} \cong \aaa\left(\begin{pmatrix} \sss\lll(n,\FF) & 0\cr 0 & 0
    \end{pmatrix}; \id,\id+\ad_B, 0\right).
$$
We will now show that 
$$\sss\nnn\aaa(n,\FF)\not\cong\aaa\left( \sss\lll(n,\FF) ; \id,\id, 0\right)\cong \aaa\left(\begin{pmatrix} \sss\lll(n,\FF) & 0\cr 0 & 0
    \end{pmatrix}; \id,\id, 0\right). 
$$
In view of Corollary~\ref{cor.class} the affgebras are isomorphic if and only if there exists $Q\in \begin{pmatrix} \sss\lll(n,\FF) & 0\cr 0 & 0
    \end{pmatrix}$ such that, for all $X\in \begin{pmatrix} \sss\lll(n,\FF) & 0\cr 0 & 0
    \end{pmatrix}$, $[B+Q,X]=0$. Since $Z(\sss\lll(n,\FF)) = \FF I_n$, and  $B$ has  $1$ as the $(n+1,n+1)$-entry, there is no $Q$ such that 
    $$
    B+Q\in \begin{pmatrix} Z(\sss\lll(n,\FF)) & 0\cr 0 & 0
    \end{pmatrix}.
    $$
    Consequently,
    $$
    \sss\nnn\aaa(n,\FF)\not\cong\aaa\left( \sss\lll(n,\FF) ; \id,\id, 0\right). 
    $$

\section{An application: central extensions of Lie affgebras}\label{sec.appl}
Let $\ggg$ be a Lie algebra. Recall that an antisymmetric bilinear form $\pi: \ggg\times \ggg \to \FF$ is  a {\em 2-cocycle} (in the Chevalley-Eilenberg complex \cite{CheEil:coh}) provided
\begin{equation}\label{2-coc}
    \pi(a,[b,c]) + \pi(b,[c,a])+\pi(c,[a,b]) =0,
\end{equation}
for all $a,b,c\in \ggg$. Given a 2-cocycle $\pi$ on $\ggg$ one can define the {central extension} of $\ggg$ by $\pi$ as a vector space $\ggg\oplus \FF\langle z\rangle$ with the Lie bracket, for all $a,b\in \ggg$ and $\alpha,\beta \in \FF$,
$$
[a+\alpha z, b+\beta z] = [a,b]+\pi(a,b)z;
$$
see \cite{CheEil:coh}.
The central extension by a 2-cocycle $\pi$ is denoted by $\ggg(\pi)$.  In this section we develop central extensions of Lie affgebras, by exploring the structure of $\aaa(\ggg;\kappa,\lambda,s)$. 

\begin{lemma}\label{lem.c.ext}
    Let $\aaa$ be a Lie affgebra, and let $\omega: \aaa\times\aaa\to\FF$ be a bi-affine map such that, for all $a,b,c\in \aaa$
    \begin{subequations}\label{aff.coc}
      \begin{equation}\label{antisym.coc}
        \omega(a,b) - \omega(a,a) +\omega(b,a) = \omega(b,b),
    \end{equation}
    \begin{equation}\label{Jacobi.coc}
         \omega(a,\{b,c\})-\omega(a,\{a,a\})+ \omega(b,\{c,a\})- \omega(b,\{b,b\}) + \omega(c,\{a,b\}) = \omega(c,\{c,c\}).
    \end{equation}
        \end{subequations}
    Then the product affine space $\aaa\times \FF$ is a Lie affgebra with the bracket
    \begin{equation}\label{aff.coc.bracket}
        \{(a,\alpha),(b,\beta)\} = \left(\{a,b\},\omega(a,b)\right),
    \end{equation}
    for all $a,b\in \aaa$ and $\alpha,\beta \in \FF$. We denote this Lie affgebra by $\aaa(\omega)$ and call it a {\em cocycle extension} of $\aaa$.
\end{lemma}
\begin{proof}
    The statement follows immediately from the observation that equations \eqref{aff.coc} express the affine antisymmetry and the Jacobi identity of the second factor in $\aaa\times \FF$.
\end{proof}

\begin{proposition}\label{prop.coc}
    Let $\aaa=\aaa(\ggg;\kappa,\lambda,s)$ and $\delta = \lambda-\kappa$. Then $\omega:\aaa\times\aaa\to \FF$ satisfies conditions \eqref{aff.coc} if and only if there exist a 2-cocycle $\pi$ on $\ggg$, linear forms  $\rho,\sigma$ on $\ggg$, and $\tau\in \FF$ such that
    \begin{equation}\label{omega.pi}
        \omega(a,b) = \pi(a,b)+\rho(a)+\sigma(b)+\tau
    \end{equation}
    and
    \begin{equation}\label{coc.deriv}
        \sigma([a,b]) = \pi(a,\delta(b))+\pi(\lambda(a),b).
    \end{equation}

    Furthermore,
    \begin{equation}\label{iso.omega.pi}
        \aaa(\omega)\cong \aaa(\ggg(\pi);\hat\kappa,\hat\lambda,\hat s),  
    \end{equation}
    where
    \begin{equation}\label{data.omega.pi}
        \hat\kappa(a+\alpha z) = \kappa(a)+ (\rho(a)+\sigma(a))z, \qquad \hat\lambda(a+\alpha z) = \lambda(a)+ \sigma(a)z,\quad \hat s = s+\tau z.
    \end{equation}
\end{proposition}
\begin{proof}
    The proof is analogous to the proof of Theorem~\ref{thm.alg.aff}. The maps $\pi,\rho,\sigma$ and the element $w$ are uniquely determined by setting
    $$
    \rho(a) =\omega(a,0)-\omega(0,0) ,\quad \sigma(a) = \omega(0,a)- \omega(0,0), \quad  \tau= \omega(0,0),
    $$
    and 
    $$
    \pi(a,b) = \omega(a,b) -\omega(a,0)-\omega(0,b)+\omega(0,0).
    $$
    The map $\pi$ is skew-symmetric if and only if $\omega$ satisfies condition \eqref{antisym.coc}.

    Next,
    $$
    \begin{aligned}
        \omega(a,\{b,c\}) - \omega(a,\{a,a\}) =&\; \pi(a,[b,c])-\pi(a,\delta(b))+\pi(a,\lambda(c)) +\sigma([b,c])  \\
        &- \pi(a,\lambda(a))+ \pi(a,\delta(a))
        -\sigma\lambda(a) +\sigma\lambda(c) +\sigma\delta(a)-\sigma\delta(b).
    \end{aligned}
     $$
     Hence \eqref{Jacobi.coc} is satisfied if and only if
     $$
     \begin{aligned}
         0 =&\; \pi(a,[b,c])-\pi(a,\delta(b))+\pi(a,\lambda(c))+\sigma([b,c]) - \pi(a,\lambda(a))+ \pi(a,\delta(a))\\
        & +\pi(b,[c,a])-\pi(b,\delta(c))+\pi(b,\lambda(a))+\sigma([c,a]) - \pi(b,\lambda(b))+ \pi(b,\delta(b))
        \\
        & +\pi(c,[a,b])-\pi(c,\delta(a))+\pi(c,\lambda(b))+\sigma([a,b]) - \pi(c,\lambda(c))+ \pi(c,\delta(c)).
     \end{aligned}$$
     By step-by-step elimination, setting first $c=0$ and then also $b=0$, we arrive at the necessity of the 2-cocycle condition for $\pi$ and \eqref{coc.deriv}. The sufficiency is clear.

     The isomorphism \eqref{iso.omega.pi} and the structure data \eqref{data.omega.pi} are obtained by comparing the expression for the Lie bracket \eqref{aff.coc.bracket} with the form of $\omega$ in \eqref{omega.pi} taken into account and the general formula for the Lie bracket in a Lie affgebra with the fibre $\ggg(\pi)$ in Theorem~\ref{thm.alg.aff}.
\end{proof}
\begin{remark}\label{rem.coc}
    When evaluated at $b=a$, the condition \eqref{coc.deriv} implies that $\pi(a,\kappa(a)) = 0$, for all $a\in \ggg$. This puts constraints on the possibility of extending $\aaa$  with a given cocycle $\pi$ on $\ggg$. As an illustration of this consider an extension of the two-dimensional Lie algebra $\ggg$ generated by $q$ and $p$ by the cocycle $\pi(q,p) =1$. The resulting Lie algebra $\ggg(\pi)$ is the Heisenberg algebra, generated by $q,p,z$ with brackets
    $$
    [q,p]=z, \qquad [q,z]=[p,z]=0.
    $$
    As discussed in Example~\ref{sec.ex.ab}, any $\kappa,\lambda,s$ give rise to a Lie affgebra $\aaa = \aaa(\ggg;\kappa,\lambda,s)$. On the other hand, only maps $\delta$ and $\lambda$, and in consequence $\kappa$ that are diagonal in the basis $\{q,p\}$  satisfy \eqref{coc.deriv}.
\end{remark}

\begin{example}\label{ex.simple.coc}
    The isomorphisms classes of cocycle extensions of a Lie algebra $\ggg$ are classified by the second cohomology group of $\ggg$, \cite[Theorem~26.2]{CheEil:coh}. In case $\ggg$ is a finite-dimensional simple Lie algebra, this group is trivial by \cite[Theorem~21.1]{CheEil:coh}, and hence -- up to isomorphism -- we can choose $\pi =0$. Since $\sigma$  in Proposition~\ref{prop.coc} must satisfy \eqref{coc.deriv}, $\sigma([a,b]) =0$, for all $a,b\in \ggg$. However, $[\ggg,\ggg] = \ggg$, and so $\sigma=0$. Thus up to coboundaries in the cohomology of $\ggg$, 
    $$
    \omega(a,b) = \rho(a) +\tau.
    $$
    Observing that, $\hat\kappa(z) =0$, one can eliminate $\tau$ from the expression for $\hat s$ in \eqref{data.omega.pi} by setting $q=-\tau z$ in \eqref{ss'}. This will not change $\hat\lambda$ by the centrality of $z$ in $\ggg(\pi)$. We conclude that all extensions of Lie affgebras $\aaa$ with simple finite dimensional fibres $\ggg$ are isomorphic to $\aaa(\ggg(0);\hat\kappa,\hat\lambda, \hat s)$, where
    \begin{equation}\label{isom.simp.ext}
        \hat\kappa(a+\alpha z) = \kappa\, a+\rho(a)z, \quad \hat\lambda (a+\alpha z)= \kappa\, a, \quad \hat s =s,
    \end{equation}
    for some $\kappa\in \FF$, $s\in \ggg$ and a linear functional $\rho$ on $\ggg$. 
\end{example}

\section*{Acknowledgements}
We would like to thank Dietrich Burde and Konstanze Rietsch for helpful discussions and suggestions. 

The research  of Tomasz Brzeziński and Krzysztof Radziszewski is partially supported by the National Science Centre, Poland, through the WEAVE-UNISONO grant no.\ 2023/05/Y/ST1/00046.

\end{document}